\documentclass[12pt]{amsart}

\usepackage{a4wide}

\usepackage{amscd}
\usepackage{amssymb}
\usepackage{amsthm}
\usepackage{amsfonts}
\usepackage{amsmath}
\usepackage{latexsym}

\usepackage{epsfig}
\usepackage{graphicx}
\usepackage{epstopdf}
\usepackage{hyperref}
\usepackage[
backend=biber,
sorting=nyt,
style=numeric
]{biblatex}

\newtheorem{prop}{Proposition}
\newtheorem{theorem}{Theorem}
\newtheorem{lemma}{Lemma}
\newtheorem{assumption}{Assumption}

\newtheorem{definition}{Definition}

\newcommand{\R}{\ensuremath{\mathbb{R}}}
\newcommand{\E}{\ensuremath{\mathbb{E}}}
\newcommand{\p}{\ensuremath{\mathbb{P}}}

\def\e{{\mathrm{e}}}

\def\diag{\mathop{\rm diag}}

\allowdisplaybreaks

\title{Multivariate self-exciting jump processes with applications to financial data}

\author{Heidar Eyjolfsson and Dag Tj\o stheim}

\begin{document}

\begin{abstract}
The paper discusses multivariate self- and cross-exciting processes. We define a class of multivariate point processes via their corresponding stochastic intensity processes that are driven by stochastic jumps. Essentially, there is a jump in an intensity process whenever the corresponding point process records an event. An attribute of our modelling class is that not only a jump is recorded at each instance, but also its magnitude. This allows large jumps to influence the intensity to a larger degree than smaller jumps.  We give conditions which guarantee that the process is stable, in the sense that it does not explode, and provide a detailed discussion on when the subclass of linear models is stable. Finally, we fit our model to financial time series data from the S\&P 500 and Nikkei 225 indices respectively. We conclude that a nonlinear variant from our modelling class fits the data best. This supports the observation that in times of crises (high intensity) jumps tend to arrive in clusters, whereas there are typically longer times between jumps when the markets are calmer. We moreover observe more variability in jump sizes when the intensity is high, than when it is low.
\end{abstract}

\maketitle

\section{Introduction}

This paper discusses a class of multivariate self-exciting and cross-exciting processes. A univariate self-exciting process can be viewed as a counting process which counts some number of events which have occurred at any given point in time. The standard Poisson process is an example of a counting process, with memoryless interarrival times, meaning that its interarrival times are exponentially distributed. The Poisson process is a suitable model in the case when the intensity of the counting process is constant, i.e. when the arrival of an event does not influence the arrival of subsequent events. This property of having identical and independently distributed interarrival times is however not always realistic. Indeed in many applications of counting processes the evidence suggests that the arrival times of events are far from being independent and identically distributed. As an example one can e.g. consider the arrivals of earthquakes and its aftershocks, the occurrence of individuals contracting flu in a given population, the occurrence of crises in financial markets. What these examples have in common is that events tend to arrive in clusters, as opposed to being independently scattered over time. This is an attribute of the class of self-exciting processes. Indeed, the class of self-exciting processes allow events to \emph{excite} the counting process intensity, which can in turn lead to this aforementioned clustering of events.

The univariate self-exciting model can be extended to a multivariate self- and cross-exciting process. Under this model events are classified as belonging to specific components, and excitation can occur both within a distinct component (self-excitation), but also between different components (cross-excitation). An example is the occurrence of shocks or crises across distinct financial markets. 

Recently, there has been a surge of interest in counting and point process methodology in time series modelling. In discrete time models the counts may be described by a Poisson-like structure, whose intensity varies in time according to a time series model, most often an autoregressive model. The resultant model gives a GARCH-like structure for an integer time series.  A relatively early case of this is the Poisson autoregression in  Fokianos et al. \cite{FRT09}.  More recent updates with many references are contained in Davis et al. \cite{DHLR15}, Fokianos et al. \cite{FSTD20}, and Debaly and Truquet \cite{DT21}.

Often events occur in continuous time at irregular time points. A class of point processes in continuous time is the Hawkes processes, originally introduced by Hawkes \cite{Ha71a, Ha71b}, but having experienced a strong recent revival in modeling of financial and other data. Some selected references are Embrechts et al. \cite{ELL11} , A\"{i}t Sahalia et al.  \cite{Ait10}, Bormetti et al., \cite{BCTCML15} and Swishchuk et al. \cite{S21}.

In this paper, extending results in Eyjolfsson and Tj{\o}stheim \cite{EyTj18}, we introduce a multivariate self-exciting and cross-exciting jump process. It is based on a continuous time point process model, where the intensity varies in time according to a stochastic differential equation (SDE), and the jump-size probability distribution may depend on the value of the intensity immediately before each jump. This generalizes (and makes more realistic) the aforementioned GARCH-type integer time series model. There  are also connections to the Hawkes processes and the conditional duration models of Engle and Russel \cite{ER98}.

We will mainly be concerned with two problems for the class of jump processes introduced. First, to be able to use the model, one has to secure that the model is stationary, that it does not explode. This problem 
has been given extensive coverage in the GARCH-driven integer time series; see in particular Fokianos et al \cite{FRT09,FSTD20} and Debaly and Truquet \cite{DT21a}. In the present paper we state and prove a general stability condition for our multivariate jump process using an extended infinitesimal generator of a Markov process combined with a continuous time Meyn and Tweedie \cite{MT3} theory. In its most general form the condition is valid for intensities generated by a nonlinear SDE, but as a corollary we specialize to linear SDEs, for which a more explicit condition can be stated. Second, we apply our model to the jump structure of a pair of financial indices represented by the S\&P 500 and the Nikkei 225 index. Estimated models (both linear and nonlinear) are obtained by maximum likelihood. We are able to detect both self-excitation and cross-excitation in the bivariate intensity process for this pair of indices, in particular how these excitations depend on size of jumps and how this is manifested in a lead lag relationship in jumps.

The time series we employ consists of daily quotes from the S\&P 500 and Nikkei 255 indices. We do however not fit the self-exciting model directly to the raw data. Instead, we extract the dates which have the largest returns in absolute value. Thus, in both markets, we extract a point process, which represents the time when the corresponding index experiences a large (positive or negative) movement; and we moreover extract sequences of magnitudes for the respective indices, that is the sizes of the returns. The extracted point processes thus represent the times when the indices experience a sharp increase or decline, respectively, whereas the magnitude sequences represent the size of the change. Note, that our model is special in that not only the point processes of jump-times are used for estimation, but also the magnitudes of the jumps. In fact we do highlight this point in our analysis, after fitting the models, we plot the intensity values versus the absolute values of the jump-sizes and discuss the resulting trends.

The paper is structured as follows. In section \ref{sec:SE} we formally define and discuss self- and cross-exciting processes. In particular, condititions that ensure stability are given. In section \ref{sec:Lin} we restrict our attention to linear intensity models, and discuss their stability in the context of matrices that define the corresponding intensities. In section \ref{sec:MLE} we fit linear and nonlinear versions of a bivarite intensity model to data extracted from daily quotes of the S\&P 500 index in New York and the Nikkei 255 index in Tokyo. Finally, in section \ref{sec:Conc} we give our concluding remarks.

\section{Self- and cross-exciting processes}\label{sec:SE}
A point process, $\{T_n\}_{n \geq 1}$, is a non-decreasing sequence of random variables with $T_0 = 0$.  For a general reference on point process we refer to the textbook by Daley and Vere-Jones \cite{DVJ88}. Given an integer $d \geq 1$, a $d$-dimensional point process, $\{(T_n,X_n)\}_{n \geq 1}$, is a double sequence such that $\{T_n\}_{n \geq 1}$ is a point process, and $\{X_n\}_{n \geq 1}$ is a sequence of random variables taking values in $\{1,\ldots,d\}$. Let $d \geq 1$ and 
define
\begin{equation}\label{def:N}
N_k(t) := 
\sum_{n \geq 1} 1_{\{X_n = k, T_n \leq t\}},
\end{equation}
for $k=1,\ldots,d$. If $N(t) := (N_1(t),\ldots,N_d(t))^\top$, then the vector $N(t)$ is the counting process associated to the multivariate point process, since it counts the number of incidents which have occured in each component up to time $t \geq 0$. We consider each component $k=1,\dots,d$ as a (univariate) point process and identify each component with its counting process and let
\begin{equation*}
\mathcal{F}_t^{N_k} := \sigma\{N_k(s) : 0 \leq s \leq t\},
\end{equation*}
where $t \geq 0$ and $k=1,\ldots,d$. Suppose that the point process $N_k(t)$ is adapted to a filtration $\{\mathcal{F}^k_t\}$, with $\mathcal{F}_t^{N_k} \subset \mathcal{F}^k_t$ and $1 \leq k \leq d$, and suppose that $N_k(t)$ admits a $(\p,\mathcal{F}^k_t)$-optional intensity $\lambda_k(t) \in \R_+^d$ (i.e. that $\lambda_k(t)$ is measurable with respect to the smallest $\sigma$-algebra on $\R_+ \times \Omega$ that make all c\`adl\`ag, adapted processes measurable) in the sense of Br\'emaud \cite{Br81}, that is, 
\begin{equation}\label{duality}
\E\left[\int_0^\infty f(s)dN_k(s) \right] = \E\left[ \int_0^\infty f(s)\lambda_k(s) ds \right],
\end{equation}
holds for all predictable $f: \Omega \times \R_+  \to [-\infty,\infty]$. In what follows, we define a multivariate self-exciting jump process to be a $d$-dimensional process $N(t) := (N_1(t),\ldots,N_d(t))^\top$, where each component is given by \eqref{def:N} and the vector of intensities $\lambda(t) = (\lambda_1(t),\ldots,\lambda_d(t))^\top$ is such that \eqref{duality} holds for each $k=1\ldots,d$, and $\lambda(t)$ is specified below. Let
\begin{equation}
U_k(t) = \sum_{n=1}^{N_k(t)} Y_n^k1_{\{X_n=k\}},
\end{equation}
for $k=1,\ldots,d$, where $\{Y_n^k\}$ is a family of non-negative random variables, and $X_n$ is as above. Suppose furthermore that $U(t) = (U_1(t),\ldots,U_d(t))^\top$, and consider the $d$-dimensional counting process \eqref{def:N}, with $d$-dimensional intensity process given by 
\begin{equation}\label{def:lambdaSDE}
d \lambda(t) = \mu(t,\lambda(t))dt + BdU(t),
\end{equation}
$\lambda(0) \geq 0$, where $\mu:\R^{d+1} \to \R^n$ is a measurable function, and $B \in \R^{d \times d}$ is a matrix. We furthermore denote by $\nu_1(\lambda,dx), \ldots, \nu_d(\lambda,dx)$ the jump-size distributions of the $d$ distinct components, i.e. $A \mapsto \nu_k(\lambda,A)$ is the probability distribution of the jump-sizes $Y_n^k$, for all $n \geq 1$, where $\lambda \in \R^d$ denotes the current value of the intensity process and $A \in \mathcal{B}(\R)$ is a Borel set. Thus, the jump part of the above intensity equation depends on the current value of the intensity process, which makes the equation \eqref{def:lambdaSDE} Markovian.  Here the drift function, $\mu$, determines the behaviour of the intensity process between jumps, and the matrix $B$, together with the jump-size distributions, characterizes the effects jumps have on the intensity in the sense that $b_{jk}$ describes the influence a jump in $N_k(t)$ has on the intensity component $\lambda_j(t)$. Thus, in particular $b_{jk} = 0$ means that the jump has no influence, and $b_{jk} = 0$ for all $j \ne k$, means that distinct components do not excite each other.

\begin{definition}
A multivariate SDE-driven self-exciting jump process is a multivariate process $N(t) = (N_1(t),\ldots,N_d(t))^\top$, such that for each $k=1,\ldots,d$, $N_k(t)$ admits the equation \eqref{def:N}, where $\{(T_n,X_n)\}$ is a $d$-dimensional point process, and the counting process $N_k(t)$ has a $(\p,\mathcal{F}_t^{N_k})$-optional intensity $\lambda_k(t)$, where $\lambda(t) = (\lambda_1(t),\ldots,\lambda_d(t))^\top$ is given by \eqref{def:lambdaSDE}, with jump-sizes, $\{Y_n^k\}$, which follow distributions $\nu_1(\lambda(T_n-),\cdot),\ldots,\nu_d(\lambda(T_n-),\cdot)$, respectively, where $\{\nu_k (\lambda,\cdot)\}_{\lambda > 0, k=1,\ldots,d}$ is a family of probability distributions.
\end{definition}
We stress that the notation $t-:= \lim_{s\uparrow t} s$ denotes the the left-limit of $t$, and thus the value of the intensity $\lambda(t)$ immediately before a jump is a parameter in the jump-size distribution. As an example of intensity models of the type \eqref{def:lambdaSDE}, we shall study the linear model (see equation \eqref{ex:linear}) in some detail in section \ref{sec:Lin}, and in section \ref{sec:MLE}, we fit both linear and non-linear models to data, and discuss the observed dependence between the intensity and jump-sizes.  

In what follows we derive the form of the extended generator of the intensity process \eqref{def:lambdaSDE}. Define $\mathcal{D}(\mathcal{A})$ to be the set of measurable functions $f: \R^d \to \R$ such that there exists a measurable function $\psi: \R^d \to \R$ and the process 
$$
C_f(t) = f(\lambda(t)) - f(\lambda(0)) - \int_0^t \psi(\lambda(s))ds
$$ 
is a local martingale with respect to the filtration generated by $\lambda(t)$, i.e. $\mathcal{F}_t = \sigma\{\lambda(s) : s \leq t\}$, under the probability measure $\p_\lambda$, $\lambda \in \R^d$, induced by the transition function of the Markov process $\lambda(t)$, with $\lambda(t)=\lambda$. Write $\psi = \mathcal{A} f$, and call $(\mathcal{A},\mathcal{D}(\mathcal{A}))$ the \emph{extended generator} of a Markov process $\lambda(t)$. For more details we refer to Davis \cite{Da93} and Eyjolfsson and Tj\o stheim \cite{EyTj18} for the univariate version of this process.
Let $e_1, \ldots, e_d$ be the standard orthonormal basis of $\R^d$, i.e. $e_k := (\delta_{1k},\ldots,\delta_{dk})^\top$is a vector of Kronecker delta's, in which the $k$th component is equal to $1$, and the remaining entries are equal to zero. Henceforth we moreover let $\langle x, y \rangle := x^\top y$ denote the dot product in $\R^d$. We show that our general process has the extended generator 
\begin{equation}
\label{def:gen}
(\mathcal{A}f)(\lambda) = \langle \mu(\cdot,\lambda), \nabla f(\lambda)\rangle + \langle \lambda, \mathcal{J}f(\lambda) \rangle,
\end{equation}
where $\mathcal{J}f(\lambda) = (\mathcal{J}_1 f(\lambda),\ldots, \mathcal{J}_d f(\lambda))^\top$, with
\begin{equation*}
\mathcal{J}_kf(\lambda) = \int (f(\lambda + b_k x) - f(\lambda))\nu_k(\lambda,dx),
\end{equation*}
and $b_k := B e_k$ ($b_k$ is the $k$th column vector of $B$), for $k=1,\ldots,d$. In the sequel we shall employ the notation $\E_\lambda = \E[\cdot |\lambda(0)=\lambda]$. The extended generator verifies the so called \emph{Dynkin formula}, which states that
$$
\E_\lambda[f(\lambda(t))] = f(\lambda) + \E_\lambda\left[\int_0^t \mathcal{A} f(\lambda(r))dr \right],
$$
for all $f \in \mathcal{D}(\mathcal{A})$.
\begin{prop}\label{prop:Dom}
If for any $t > 0$ and a given measurable function, $f:\R_+^d \to \R$, it holds that the map
$$
\lambda \mapsto \int f(\lambda + b_k x) \nu_k(\lambda,dx)
$$
is measurable for $k=1,\ldots,d$, and
\begin{equation}
\label{def:phiCond}
\E_\lambda \left[\int_0^t \sum_{k=1}^d \lambda_k(s) \int \{f(\lambda(s-)+b_k x) - f(\lambda(s-)\}\nu_k(\lambda(s-),dx) \,ds \right] < \infty
\end{equation}
then $f$ is in the domain of the extended generator of $\lambda(t)$, $f \in \mathcal{D}(\mathcal{A})$, where $\mathcal{A}$ is given by \eqref{def:gen}.
\end{prop}

\begin{proof}
By It\^o's lemma it holds that
\begin{align*}
f(\lambda(t)) - f(\lambda(0)) &= \int_0^t \langle \mu(s,\lambda(s)), \nabla f(\lambda(s)) \rangle ds +  \sum_{0 < s \leq t} \left\{f(\lambda(s)) - f(\lambda(s-)) \right\}.
\end{align*}
Now, using that $N(t) - \int_0^t \lambda(s)ds$ is a martingale,

\begin{align*}
&\E_{\lambda}\left[\sum_{0 < s \leq t} \left\{f(\lambda(s)) - f(\lambda(s-)) \right\}\right] \\
&= \E_\lambda \left[\sum_{k=1}^d \int_0^t  \int \{f(\lambda(s-)+b_k x) - f(\lambda(s-)\}\nu_k(\lambda(s-),dx) N_k(ds) \right] \\
&= \E_\lambda \left[\sum_{k=1}^d \int_0^t  \lambda_k(s) \int \{f(\lambda(s-)+b_k x) - f(\lambda(s-)\}\nu_k(\lambda(s-),dx) \,ds \right],
\end{align*}
where we have exploited the fact that 
$$
s \mapsto \int \left\{f(\lambda(s-)+b_k x) - f(\lambda(s-)) \right\}\nu_k(\lambda(s-),dx),
$$
is predictable, and that a stochastic integral of a predictable process with respect to a martingale is a martingale. It follows that for a function $f$ which fulfils \eqref{def:phiCond}, the process 
$$
t \mapsto f(\lambda(t)) - f(\lambda(0)) - \int_0^t \mathcal{A}f(\lambda(s))ds
$$ 
is a zero-mean martingale, and thus the proof is completed using the definition of the extended generator.
\end{proof}

We remark that strictly speaking the time dependence of the drift function, $\mu$, means that we should consider $t \mapsto (t,\lambda(t))$ as a Markov process  and add $\partial/\partial t f(t,\lambda)$ to the extended generator. To simplify the notation we  circumvent this in our definition. We moreover note that the generator \eqref{def:gen} is a linear operator on its domain, $\mathcal{D}(\mathcal{A})$. 


Having identified the extended generator of our class of processes, we proceed to use it to analyse some of the class properties. To that end, we first of all notice that our class of self-exciting processes is a piecewise deterministic process (PDP) in the sense of Davis \cite{Da93}. We adopt the following regularity conditions.

\begin{assumption}
It holds that
\begin{enumerate}
\item[i)] $(t,\lambda) \mapsto \mu(t,\lambda)$ is Lipschitz continuous, and the solution of $y' = \mu(t,y)$, $y(0) = \lambda \geq \lambda_0$, does not explode in finite time.
\item[ii)] $\nu_k:[0,\infty) \to \mathcal{P}(\R)$ (the set of probability measures on $\R$) is a measurable function such that $\nu_k(\lambda,\{\lambda\}) = 0$ for all $\lambda \geq 0$ and $k=1,\ldots,d$. 
\item[iii)] The map $\lambda \mapsto \int_{\lambda_0}^\infty f(x) \nu_k(\lambda,dx)$ is continuous for continuous bounded $f$ and $k=1,\ldots,d$.
\end{enumerate}
\end{assumption}

Together with the non-explosion assumption, the first two of the above assumptions are the so-called ``standard conditions'' of Davis \cite{Da93}, which ensure a certain regularity structure on the class of PDP processes. The first one of these concerns the deterministic $\mu$ function, which governs the behaviour of the intensity function between jumps. By requiring Lipschitz continuity, and excluding explosions, we ensure that the process behaves like a deterministic non-explosive Markov process between jumps. The second condition states the measurability of the family of jumps-size distributions, and that we can almost surely detect jumps. The above assumption moreover ensures that our process class is a so-called Borel right process (see Theorem 27.8 in Davis \cite{Da93}). Finally, the continuity assumption of point three ensures together with the non-explosive property that our process class fulfills the Feller property, i.e. that the map $\lambda \mapsto P_t f(\lambda) := \E[f(\lambda(t))|\lambda(t)=\lambda]$ is bounded and continuous  if $f$ is bounded and continuous for $t \geq 0$.

The Markov process $\lambda(t)$ is said to be \emph{$\phi$-irreducible} if $\phi$ is $\sigma$-finite and 
$$
\E_\lambda\left[ \int_0^\infty 1_{\{ \lambda(t) \in A \}} dt \right] > 0
$$
whenever $\phi(A) > 0$, for all $\lambda \geq \lambda_0$. The following stability result employs the form of the generator to analyse the stability properties of the intensity process. It turns out that under certain assumptions on the generator the intensity process \eqref{def:lambdaSDE} is asymptotically stable. Given a signed measure on $\mathcal{B}(\R^d)$ and $f \geq 1$, write $\|\mu\|_f := \sup_{|g| \leq f} \left|\int g d\mu \right|$.
\begin{theorem}\label{thm:MT}
Suppose that $f$ is a norm on $\R^d$, or a monotone increasing and unbounded function of a norm on $\R^d$. If there exist constants $C_1 > 0$ and $C_0 \in \R$ such that
\begin{equation}\label{gen:cond}
\mathcal{A}f(\lambda) \leq C_0 - C_1 f(\lambda) 
\end{equation}
for all $\lambda$, then an essentially unique finite invariant measure, $\pi$, exists and $\lambda(t)$ is moreover geometrically ergodic, i.e. there exist $\beta < 1$, $K < \infty$ such that
$$
\|P_t(\lambda,\cdot) - \pi\|_f \leq Kf(\lambda)\beta^t,
$$
where $P_t(\lambda,\cdot) = \p_\lambda(\lambda(t) \in \cdot)$.
\end{theorem}
\begin{proof}
Let $\phi$ be a finite measure on $\R^d$ which is supported on $[\lambda_0^1,\infty) \times \cdots \times [\lambda_0^d,\infty)$, and absolutely continuous with respect to the Lebesgue measure on $\R^d$. Then $\lambda(t)$ is $\phi$-irreducible (this also holds for a sampled chain). It follows by the Feller property and Theorem 3.4 in Meyn and Tweedie  \cite{MT1} that all compact subsets of a skeleton chain are petite. So the result follows directly from Theorem 6.1 in Meyn and Tweedie \cite{MT3}.
\end{proof}

\section{The linear drift model}\label{sec:Lin}

In this section we consider the linear case, i.e. the case when $\mu(\lambda) = A( \lambda - \lambda_0)$, for a vector $\lambda(0) = \lambda_0 = (\lambda_0^1,\ldots,\lambda_0^d)^\top \geq 0$, and a matrix $A$. In this case we may write the intensity dynamics as 
\begin{equation}\label{ex:linear}
d\lambda(t) = A(\lambda(t) - \lambda_0)dt + BdU(t).
\end{equation}
This linear intensity is a continuous time analogue to the discrete time autoregressive model discussed by Fokianos et al. \cite{FRT09}.
\begin{prop}
If $A$ in the linear model \eqref{ex:linear} is diagonalisable, then 
$$
\lambda(t) = \lambda_0 + \int_0^t \e^{A(t-s)} B dU(s).
$$
\end{prop}
\begin{proof}
Suppose $A = EDE^{-1}$, where $E$ is a matrix consisting of the eigenvectors of $A$ and $D = \diag(\alpha_1,\ldots,\alpha_n)$ is a diagonal matrix with the eigenvalues of $A$ on the diagonal, then letting $X(t) := E^{-1}(\lambda(t) - \lambda_0)$ and $Z(t) := E^{-1}BU(t)$, it follows that 
$$
dX(t) = DX(t)dt + dZ(t), 
$$
where $X(0) = 0$. Now for $k=1,\ldots,n$ it follows by It\^o's lemma that 
$$
X_k(t)\e^{-\alpha_k t} = -\int_0^t \alpha_k X_k(s)\e^{-\alpha_k s} ds + \int_0^t \e^{-\alpha_k s} dX_k(s) = \int_0^t \e^{-\alpha_k s} dZ_k(s),
$$
from which it follows that 
$$
X(t) = \int_0^t \e^{D(t-s)} dZ(s),
$$
and thus since $\e^{A(t-s)} = E\e^{D(t-s)}E^{-1}$ it holds that 
$$
\lambda(t) = \lambda_0 + \int_0^t \e^{A(t-s)} B dU(s).
$$
\end{proof}
According to the this Proposition, if $t_1 > t_0$ we may write
\begin{align*}
\lambda(t_1) - \lambda_0 &= \e^{A(t_1-t_0)}(\lambda(t_0) - \lambda_0) + \int_{t_0}^{t_1}\e^{A(t_1 - s)}BdU(s).
\end{align*} 
So, for a discrete time grid $t_0 < t_1 < \cdots < t_n$, one can view $\lambda(t_k) - \lambda_0$ as a self-exciting autoregressive time series. We furthermore remark that in this case, the intensity process is consistent with the definition of a multivariate Hawkes process with an exponential kernel and stochastic jump sizes.


\begin{lemma}\label{lem:LinGen}
	Suppose that $f(\lambda) = \lambda^\alpha = \lambda_1^{\alpha_1}\cdots \lambda_d^{\alpha_d}$, where $\alpha = (\alpha_1,\ldots,\alpha_d)$ is a $d$-dimensional vector of non-negative integers such that $\alpha_1 + \cdots + \alpha_d = n$. Then if $J(\lambda)$ is the $d$-dimensional diagonal matrix with $\int x \nu_k(\lambda,dx)$, $k=1,\ldots,d$, on the diagonal, then under the linear model \eqref{ex:linear} it holds that
	\begin{align*}
		\mathcal{A}f(\lambda) &= (\nabla f(\lambda))^\top M(\lambda)\lambda  +  \sum_{m=2}^n \sum_{|\alpha| = m} \frac{D^\alpha f(\lambda)}{\alpha!} (\xi^\alpha(\lambda))^\top \lambda - (\nabla f(\lambda))^\top A\lambda_0 ,
	\end{align*}
	where  $M(\lambda) = A + BJ(\lambda)$, and $\xi^\alpha(\lambda) = (\xi_1^{\alpha} (\lambda),\ldots,\xi_d^{\alpha}(\lambda))$, with $\xi_k^\alpha(\lambda) = \int (b_k x)^{\alpha} \nu_k(\lambda,dx)$ for $k=1,\ldots,d$ and any $d$-dimensional multi-index $\alpha$, and the sum from $m=2$ to $n$ in the final line is dropped in the case when $n<2$. 
\end{lemma}
\begin{proof}
	An application of Taylor's theorem yields: 
	\begin{align*}
		f(\lambda+b_k x) - f(\lambda) &= \sum_{0 < |\alpha| \leq n} \frac{D^\alpha f(\lambda)}{\alpha!}(b_k x)^\alpha = \sum_{m=1}^n \sum_{|\alpha| = m} \frac{D^\alpha f(\lambda)}{\alpha!}(b_k x)^\alpha.
	\end{align*}
	Hence, 
	\begin{align*}
		\mathcal{J}_k f(\lambda) &= \sum_{m=1}^n \sum_{|\alpha| = m} \frac{D^\alpha f(\lambda)}{\alpha!} \int (b_k x)^\alpha \nu_k(\lambda,dx)
	\end{align*}
	and
	\begin{align*}
		\lambda^\top \mathcal{J}f(\lambda) &=  \sum_{m=1}^n \sum_{|\alpha| = m} \frac{D^\alpha f(\lambda)}{\alpha!} \lambda^\top \xi^\alpha(\lambda) 
	\end{align*}
	The generator thus takes the form 
	\begin{align*}
		\mathcal{A}f(\lambda) &= (\nabla f(\lambda))^\top (A(\lambda-\lambda_0))  + \lambda^\top \mathcal{J}f (\lambda) \\
		&= (\nabla f(\lambda))^\top (A + B J(\lambda))\lambda  \\  
		&\ +  \sum_{m=2}^n \sum_{|\alpha| = m} \frac{D^\alpha f(\lambda)}{\alpha!} (\xi^\alpha(\lambda))^\top \lambda - (\nabla f(\lambda))^\top A\lambda_0.
	\end{align*}	
\end{proof}

\begin{lemma}\label{lem:negMat}
Suppose that $M(\lambda)$ is a matrix which depends on a parameter $\lambda \in \R^d$, and that for each $\lambda \in \R^d$ $M(\lambda) + M(\lambda)^\top$ is diagonalisable, with negative eigenvalues.
\begin{itemize}
	\item If $\sup_\lambda \gamma_1(\lambda) < 0$, where $\gamma_1(\lambda)$ denotes the largest negative eigenvalue of $M(\lambda) + M(\lambda)^\top$, then  $\lambda M(\lambda)\lambda^\top \leq -C_1\|\lambda\|_2^2$ holds for all $\lambda \in \R^d$, where $\|\lambda\|_2^2 = \lambda^\top\lambda$, and $C_1 = -\sup_\lambda \gamma_1(\lambda) > 0$.
	\item If $\inf_\lambda \gamma_0(\lambda) < 0$, where $\gamma_0(\lambda)$ denotes the smallest negative eigenvalue of $M(\lambda) + M(\lambda)^\top$, then  $\lambda M(\lambda)\lambda^\top \geq -C_0\|\lambda\|_2^2$ holds for all $\lambda \in \R^d$, where $\|\lambda\|_2^2 = \lambda^\top\lambda$, and $C_0 = -\inf_\lambda \gamma_0(\lambda) > 0$.
\end{itemize}
\end{lemma}
\begin{proof}
If $M_1(\lambda) = \frac 1 2 (M(\lambda) + M(\lambda)^\top)$ is the symmetric part of $M(\lambda)$ then $M_1(\lambda)$ can be diagonalised as $M_1(\lambda) = E(\lambda)D(\lambda)E(\lambda)^\top$, where $E(\lambda)$ is an orthogonal matrix consisting of the eigenvectors of $M_1(\lambda)$, and $D(\lambda)$ is a diagonal matrix of the corresponding eigenvalues. The column vectors of $E(\lambda)$ moreover form a basis for $\R^d$, so one can write $\lambda = E(\lambda)x$ for some $x \in \R^d$. Thus, if $\gamma_1(\lambda)$, denotes the largest negative eigenvalue of $M_1(\lambda)$, it follows that 
\begin{align*}
	\lambda^\top M(\lambda) \lambda &= \lambda^\top M_1(\lambda) \lambda = x^\top D(\lambda)x \leq \gamma_1(\lambda) \|x\|_2^2 \leq -C_1 \|\lambda \|_2^2,
\end{align*}
where we have employed the orthogonality of $E(\lambda)$, which implies that 
$$
\|x\|_2^2 = x^\top x = (E(\lambda)^\top \lambda)^\top(E(\lambda)^\top\lambda) = \lambda^\top \lambda = \|\lambda\|_2^2.
$$
The other inequality of the lemma is proved analogously.
\end{proof}

\begin{assumption}\label{Ass:Lin}
Suppose $M(\lambda) = A + BJ(\lambda)$ is a diagonalisable $d$-dimensional square matrix where $A$ and $B$ are as in \eqref{ex:linear} and $J(\lambda)$ is a $d$-dimensional diagonal matrix with 
$$
\int x \nu_k(\lambda,dx),
$$
$k=1,\ldots,d$, $\lambda \in \R^d, \lambda \geq \lambda_0$, on the diagonal. Suppose furthermore that the moments
$$
\int x \nu_k(\lambda,dx) \ \ \text{ and }\ \ \int x^2 \nu_k(\lambda,dx)
$$
are bounded as a functions of $\lambda$, all the eigenvalues of $M(\lambda)+M(\lambda)^\top$ are negative and $\sup_{\lambda \geq \lambda_0} \gamma(\lambda) < 0$, where $\gamma(\lambda)$ denotes the largest (negative) eigenvalue of the matrix $M(\lambda)+M(\lambda)^\top$.
\end{assumption}
The following proposition gives conditions under which the linear model is stable and geometrically ergodic. 
\begin{prop}\label{prop:stable}
Under Assumption~\ref{Ass:Lin}, $\lambda(t)$ is stable in the sense of Theorem \ref{thm:MT} with $f(\lambda) = \lambda^\top \lambda$, i.e. \eqref{gen:cond} holds under the linear model.
\end{prop}
\begin{proof} 
Let $f(\lambda) = \|\lambda\|_2^2/2 = \lambda^\top \lambda/2$. An application of Lemma~\ref{lem:LinGen} with the functions $f_j(\lambda) = \lambda_j^2/2$, $j=1,\ldots,d$, and the linearity of the generator gives us that 
\begin{align*}
(\mathcal{A}f)(\lambda) &= \lambda^\top M(\lambda) \lambda + \lambda^\top (w(\lambda) - A\lambda_0),
\end{align*}
where $w_k(\lambda) = f(b_k) \int x^2 \nu_k(\lambda,dx)$, $k=1,\ldots,d$. By Lemma~\ref{lem:negMat} it holds that there exists a constant $C > 0$ such that $\lambda^\top M(\lambda) \lambda \leq -C \|\lambda \|_2^2$. The proof is concluded by noting that
$$
(\mathcal{A}f)(\lambda) \leq -C \|\lambda \|_2^2 + \lambda^\top (w(\lambda) - A\lambda_0) \leq - C_1 f(\lambda) + C_2
$$
where $C_1 > 0$ and $C_2 \in \R$ are constants, since the quadratic term $\|\lambda \|_2^2$, grows faster than any linear term $\lambda^\top v$, $v \in \R^d$, and thus one can  bound $(\mathcal{A}f)(\lambda)$ with $C_2$ in a compact subset of $\R^d$.
\end{proof}

According to the preceding proposition the stability of the linear model is determined by the matrix $M(\lambda)$ of Assumption \ref{Ass:Lin}, and the jump-size moments. In what follows, we take a closer look at how this matrix determines the stability of the model in terms of norms on $\R^d$.

\begin{prop}\label{prop:FiniteNorm}
Suppose Assumption~\ref{Ass:Lin} holds, and that $\| \cdot \|$ is a norm on $\R^d$. Then  
$$
\E_\lambda[\|\lambda(t)\|] < C\left(\frac{C_2}{C_1} + \left(\lambda - \frac{C_2}{C_1}\right)\e^{-C_1t}\right)^{1/2}
$$ 
holds for all $t \geq 0$, where $C, C_1, C_2 > 0$ are constants, $C_1=-\sup_\lambda\gamma_1(\lambda)$, and $\gamma_1(\lambda)$ is the largest negative eigenvalue of $M(\lambda) + M(\lambda)^\top$.
\end{prop}
\begin{proof}
Let $f(\lambda) = \|\lambda\|_2^2$, then using the same steps as in the proof of Proposition~\ref{prop:stable} it holds that $\mathcal{A}f(\lambda) \leq -C_1\|\lambda\|_2^2 + C_2$, where $C_1, C_2 > 0$ are constants, and according to Lemma \ref{lem:negMat}, $C_1 = -\sup_\lambda\gamma_1(\lambda)$, where  $\gamma_1(\lambda)$ is the largest negative eigenvalue of $M(\lambda) + M(\lambda)^\top$. Hence, by Dynkin's lemma it holds that
 	\begin{align*}
		\E_\lambda[\|\lambda(t)\|_2^2]  &\leq \|\lambda(0)\|_2^2 + \int_0^t (-C_1\E_\lambda[\|\lambda(s)\|_2^2] + C_2)ds.
	\end{align*}
	Thus, if $z(t) := \E_\lambda[\|\lambda(t)\|_2^2] - C_2/C_1$, it follows by differentiating both sides that $z'(t) \leq -C_1z(t)$, so it follows by Grönwall's inequality that $z(t) \leq z(0)\e^{-C_1t}$, from which it follows that 
	$$
	\E_\lambda[\|\lambda(t)\|_2^2] \leq \frac{C_2}{C_1} + \left(\lambda - \frac{C_2}{C_1}\right)\e^{-C_1 t}
	$$
	for all $t \geq 0$. Now recall that all norms on finite-dimensional vector spaces are equivalent. Therefore by Jensen's (or the Cauchy-Schwarz) inequality and the equivalence of all norms on $\R^d$ (since it is finite dimensional), it follows that there exists a $C > 0$ such that
	$$
	\left(\E_\lambda[\|\lambda(t)\|]\right)^2 \leq \E_\lambda[\|\lambda(t)\|^2] \leq C\E_\lambda[\|\lambda(t)\|_2^2],
	$$
	where $\|\cdot\|$ is an arbitrary norm on $\R^d$. The result follows.
\end{proof}
Note that in Proposition \ref{prop:FiniteNorm} the value of the largest negative eigenvalue, $C_1 = -\sup_\lambda\gamma_1(\lambda)$, together with the inital value $\lambda(0)=\lambda$, determines how far on average the intensity can drift away from its stationary mean, via the $1/C_1$ terms, and the speed at which it mean-reverts,  by means of the $\e^{-C_1t}$ term. If $C_1$ is close to zero, then the intensity may drift far away from its mean value, whereas if $C_1$ is much larger than zero, then the mean reversion is much quicker. In both cases however, Proposition \ref{prop:FiniteNorm} states that the function $t \mapsto \E_\lambda[\|\lambda(t)\|]$ is bounded.


\subsection{Homogeneous jump-size distributions}

In this subsection we focus our attention on the case when the jump-size distribution of the linear model is homogeneous with respect to the intensity, i.e. when $\nu_k(\lambda,dx) = \nu_k(dx)$ holds for all $k=1,\ldots,d$. We derive formulas for the first two moments of $\lambda(t)$ in this case.

\begin{prop}\label{prop:Lin1mom}
Suppose that $\nu_k(\lambda,dx) = \nu_k(dx)$ is constant with respect to $\lambda$ for all $k=1,\ldots,d$. Then, for any $t \geq 0$, and $\lambda \geq \lambda_0$ it holds that
$$
\E_\lambda[\lambda(t)] = M^{-1}A\lambda_0 + \e^{tM}(\lambda - M^{-1}A\lambda_0),
$$
where $M = A+BJ$, $J$ is a diagonal matrix with $\int x \nu_k(dx)$, $k=1,\ldots,d$ on the diagonal. If in particular $M+M^\top$ has negative eigenvalues, then
$$
\lim_{t \to \infty} \E_\lambda[\lambda(t)] = M^{-1}A\lambda_0.
$$
\end{prop}
\begin{proof}
Suppose $f_j(\lambda) = \lambda_j$, where $1 \leq j \leq d$, then it holds by Lemma~\ref{lem:LinGen} that
\begin{align*}
	\mathcal{A}f_j(\lambda) &= (M(\lambda)\lambda - A\lambda_0)_j.
\end{align*}
Hence if, we denote the vector of first moments by $y(t) = (y_1(t),\ldots,y_d(t))^\top$, where $y_j(t) = \E_\lambda[f_j(\lambda(t))] = \E_\lambda[\lambda_j(t)]$, then it follows by Dynkin's lemma that
\begin{align}
	y(t)  &= y(0) + \int_0^t (\E_\lambda[M(\lambda(s))\lambda(s)] - A\lambda_0)ds  \label{1momentODE}\\
	&= \lambda + \int_0^t (M y(s) - A\lambda_0)ds, \nonumber
\end{align}	
where $M(\lambda) = M$ since $\nu_k(\lambda,dx) = \nu_k(dx)$ for $k=1,\ldots,d$. It follows by differentiating both sides that $y'(t) = M y(t) - A\lambda_0$, so 
\begin{align*}
\E_\lambda[\lambda(t)] &= \e^{M t}\lambda - \int_0^t \e^{(t-s) M}A ds\lambda_0 \\
&= \e^{t M}\lambda - (\e^{t M} - I) M^{-1}A \lambda_0 \\
&= M^{-1}A\lambda_0 + \e^{t M}(\lambda - M^{-1}A\lambda_0).
\end{align*}
From which our result follows.
\end{proof}


Notice that in the simplified case of the current subsection, the matrix $M$ and its eigenvalues determine wether or not $\lambda(t)$ is stable. If the eigenvalues of $M+M^\top$ are all negative then $\lambda(t)$ is stable in the sense that it does not explode. If the eigenvalues of $M+M^\top$ are close to zero (but still negative) then $\lambda(t)$ may drift further from its long term mean value, than in the case when the eigenvalues are far below zero. Clearly, these observations are congruent with the results of the Proposition \ref{prop:FiniteNorm}, and the discussion that follows after it.

\begin{prop}
Suppose that $\nu_k(\lambda,dx) = \nu_k(dx)$ is constant with respect to $\lambda$ for all $k=1,\ldots,d$. Then, if $V(t) = \E_\lambda[\lambda(t)\lambda(t)^\top]$, for any $t \geq 0$, and $\lambda \geq \lambda_0$ is constant, it holds that
$$
 V(t) = \e^{t M}\left(\lambda\lambda^\top + \int_0^t \e^{-s M}F(s)\e^{-s M^\top} ds\right)\e^{t  M^\top},
$$
where $M = A+BJ$, $J$ is a diagonal matrix with $\int x \nu_k(dx)$, $k=1,\ldots,d$ on the diagonal, $y(t) = \E_\lambda[\lambda(t)]$,  
$$
F(t) = - y(t)\lambda_0^\top A^\top - A\lambda_0(y(t))^\top  +  B^\top \Xi_2(y(t))B,
$$ 
and $\Xi_2(y)$ is a diagonal matrix with $y_k\int x^2 \nu_k(dx)$, $k=1,\ldots,d$ on the diagonal.
\end{prop}
\begin{proof}
Let $f(\lambda) = \lambda_j\lambda_k$, where $1 \leq j,k \leq n$, then if $M_j = (m_{jk})_{k=1}^d$, is the $j$th line of the matrix $M$ for $j=1,\ldots,d$ (and $A_j$ is the $j$th line $A$), it follows by Lemma~\ref{lem:LinGen} that
\begin{align*}
	\mathcal{A}f(\lambda) &= \lambda_k M_j \lambda + \lambda_j M_k \lambda - \lambda_k A_j \lambda_0 - \lambda_j A_k \lambda_0 + b_j^\top \Xi_2(\lambda) b_k \\
	&= \sum_{i=1}^d\left(\lambda_k(m_{ji} \lambda_i - a_{ji}\lambda_0^i)   + \lambda_j (m_{ki}\lambda_i - a_{ki}\lambda_0^i) + \lambda_ib_{ij}b_{ik} \int x^2 \nu_i(dx) \right).
\end{align*}
Now, if $y(t) = (\E_\lambda[\lambda_i(t)])_{i=1}^d$, it follows by Dynkin's lemma that
\begin{align*}
	\E_\lambda[\lambda_j(t)\lambda_k(t)] &= \lambda_j\lambda_k + \int_0^t \sum_{i=1}^d \E_\lambda[m_{ki}\lambda_j(s)\lambda_i(s) + m_{ji}\lambda_k(s)\lambda_i(s)]ds  \\
	& \ - \int_0^t \sum_{i=1}^d \left(a_{ki}\lambda_0^i y_j(s) + a_{ji}\lambda_0^i y_k(s) - y_i(s)b_{ij}b_{ik} J_i^2 \right)ds,
\end{align*}
where $J_i^2 = \int x^2 \nu_i(dx)$. Now, if $v_{jk}(t) = \E_\lambda[\lambda_j(t)\lambda_k(t)]$, then
\begin{align*}
	v_{jk}'(t) = \sum_{i=1}^d (m_{ki}v_{ji}(t) + m_{ji}v_{ki}(t)) - \sum_{i=1}^d (a_{ki}\lambda_0^i y_j(t) + a_{ji}\lambda_0^i y_k(t) - y_i(t)b_{ij}b_{ik}J_i^2),
\end{align*}
which in matrix form, $V(t) = (v_{jk}(t))_{j,k=1}^d$, takes the form
\begin{align*}
	V'(t) &= V(t)M^\top + MV(t) - y(t)\lambda_0^\top A^\top - A\lambda_0(y(t))^\top  +  B^\top \Xi_2(y(t))B.
\end{align*}
Hence, if $F(t) = - y(t)\lambda_0^\top A^\top - A\lambda_0(y(t))^\top  +  B^\top \Xi_2(y(t))B$, then $V'(t) = V(t)M^\top + MV(t) + F(t)$, and it holds that
\begin{align*}
	V(t) = \e^{tM}\left(\lambda\lambda^\top + \int_0^t \e^{-sM}F(s)\e^{-sM^\top} ds\right)\e^{tM^\top}.
\end{align*}

\end{proof}

Behr et al \cite{BBH19} study so-called differential Sylvester equations. The differential equation which $V(t)$ in the preceeding proof verifies is an example of a differential Sylvester formula. According to their spectral decompositions of a Sylvester operators (Lemma 3), if the matrix $M+M^\top$ has negative eigenvalues, then $\lim_{t \to \infty} V(t)$ is finite. We conclude this section with the following result on the covariance and autocovariance structure of the intensity process.


%

\begin{prop}
	Suppose that $\nu_k(\lambda,dx) = \nu_k(dx)$ is constant with respect to $\lambda$ for all $k=1,\ldots,d$. Given $\lambda \geq \lambda_0$, and $t,h \geq 0$, let $C_\lambda(t,h) = \E_\lambda[(\lambda(t+h)-\E_\lambda[\lambda(t+h)])(\lambda(t)-\E_\lambda[\lambda(t)])^\top]$ denote the covariance matrix of $\lambda(t+h)$ and $\lambda(t)$ under $\p_\lambda$. Then it holds that
	$$
	C_\lambda(t,h) = e^{hM}\left(V(t) - \e^{tM}\lambda y(t)^\top + (\e^{tM}-I)M^{-1}A\lambda_0y(t)^\top \right),
	$$
	where $M=A+BJ$, $J$ is a diagonal matrix with $\int x \nu_k(dx)$, $k=1,\ldots,d$ on the diagonal, $y(t) = \E_\lambda[\lambda(t)]$ and $V(t) = \E_\lambda[\lambda(t)\lambda(t)^\top]$. If $M+M^\top$ has negative eigenvalues, then
	$$
	\lim_{t\to \infty} C_\lambda(t,h) = \e^{hM}(\lim_{t\to \infty}V(t) - M^{-1}A\lambda_0(M^{-1}A\lambda_0)^\top ).
	$$
\end{prop}
\begin{proof}
	According the Markov property, and Proposition \ref{prop:Lin1mom} it holds that
	\begin{align*}
		\E_\lambda[\lambda(t+h)\lambda(t)^\top] &= \E_\lambda[\E[\lambda(t+h) | \lambda(t)]\lambda(t)^\top] \\
		&= \E_\lambda[(M^{-1}A\lambda_0 + \e^{hM}(\lambda(t) - M^{-1}A\lambda_0))\lambda(t)^\top] \\
		&= M^{-1}A\lambda_0y(t)^\top + \e^{hM}(V(t) - M^{-1}A\lambda_0y(t)^\top),
	\end{align*}
	where $y(t) = \E_\lambda[\lambda(t)]$ and $V(t) = \E_\lambda[\lambda(t)\lambda(t)^\top]$, respectively. By refering to Propostion \ref{prop:Lin1mom} again it follows that if $C_\lambda(t,h)$ denotes the covariance matrix of $\lambda(t)$ and $\lambda(t+h)$ under the probability measure $\p_\lambda$, it holds that
	\begin{align*}
		C_\lambda(t,h) &= M^{-1}A\lambda_0y(t)^\top + \e^{hM}(V(t) - M^{-1}A\lambda_0y(t)^\top) \\
		& \ - (M^{-1}A\lambda_0 + \e^{(t+h)M}(\lambda - M^{-1}A\lambda_0))y(t)^\top \\
		&= \e^{hM}\left(V(t) - \e^{tM}\lambda y(t)^\top + (\e^{tM}-I)M^{-1}A\lambda_0y(t)^\top \right).
	\end{align*}
	This completes our proof.
\end{proof}

\section{Modelling joint movements of historical share prices}\label{sec:MLE} In this section we give an example of how a bivariate SDE-driven self-exciting model is fitted to data. Our starting point is the bivariate time-series which consists of the stock indices in New York and Tokyo respectively, i.e. the S\&P 500 and the Nikkei 225. The stock indices measure the performance of the stock markets in the United States and Japan, respectively. We employ a bivariate self-exciting model to predict how large movements in the respective stock indices influence the likelihood of a large movement in the other stock index. The plots are produced with R, and the maximum likelihood estimation is done with MATLAB.

The financial time series were downloaded from \url{finance.yahoo.com}, they are plotted in figure \ref{fig:ind}.
\begin{figure} 
\centering
\includegraphics[width=6.5in, height=2.5in]{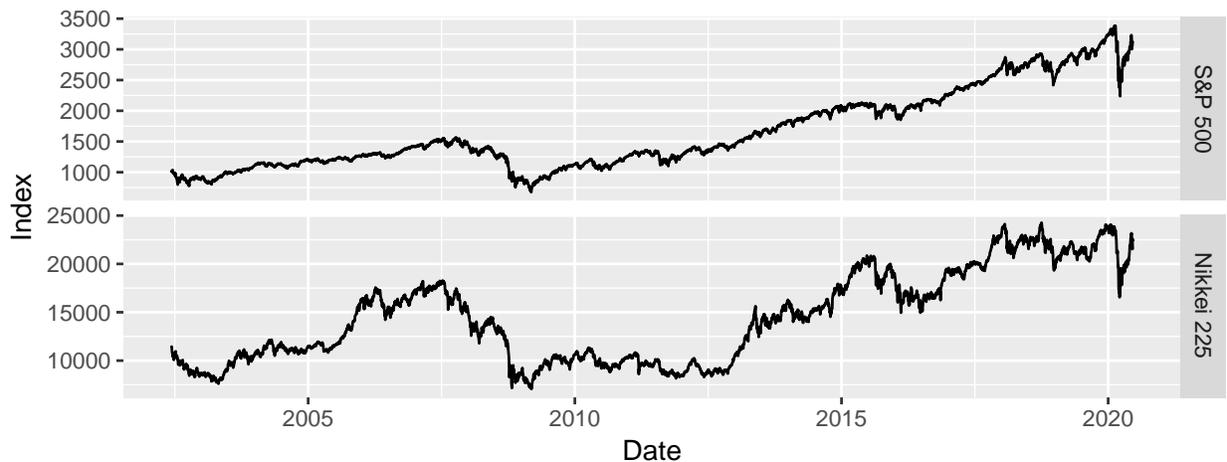}
\caption{Historical data of the S\&P 500 (upper panel) and the Nikkei 225 (lower panel) stock indices.}
\label{fig:ind}
\end{figure}
The bivariate point process data where each point is accompanied by a jump-size is then extracted from the time series. We extract the jump-times and jumps-sizes from the corresponding series of log-returns.
Basically, we say that a jump occurs in a specific component when the absolute value of a log-return is larger than a fixed threshold which we set equal to $0.025$. We use the same threshold for both series, thus the series with the higher volatility will produce more jumps than the series with the lower volatility. In our case this means that more jumps are extracted from the Nikkei 225 time series than the S\&P 500 time series.  
\begin{figure} 
\centering
\includegraphics[width=6.5in, height=2.5in]{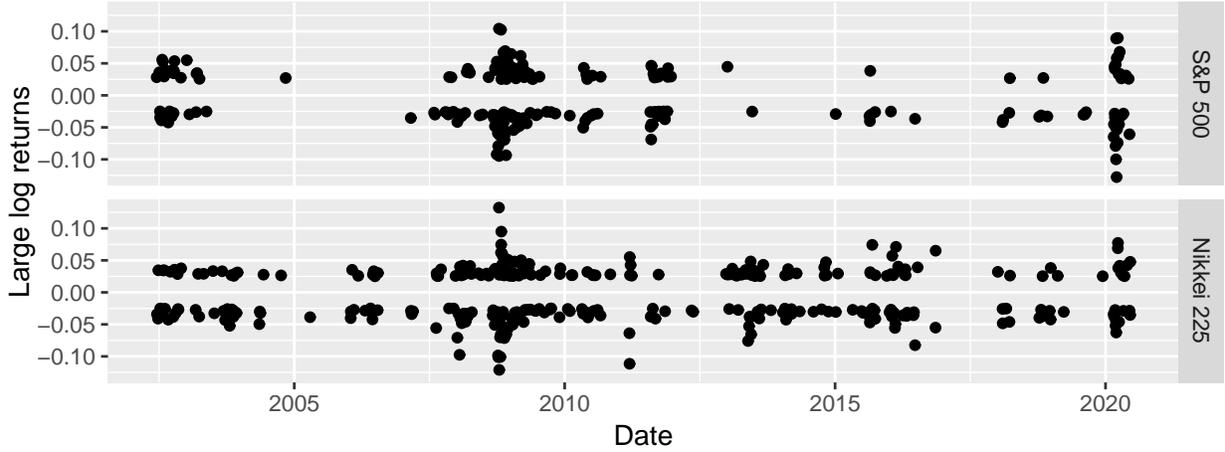}
\caption{Log returns the S\&P 500 (upper panel) and the Nikkei 225 (lower panel) stock indices, which are larger than $0.025$ in absolute value.}
\label{fig:lrLarge}
\end{figure}
Furthermore, not only an occurrence of a jump is recorded, but also its magnitude. Thus, jumps of large magnitude contribute more to the excitation of the intensity than smaller jumps. In figure \ref{fig:lrLarge} the extracted jumps of the respective log-return time-series are displayed, and in figure \ref{fig:lrLargeHists} positive and negative histograms of jumps which are larger than $0.025$ in absolute value are plotted. Notice that the largest positive and and negative log-returns in both indices appear during the financial crisis of 2008 and the corona virus pandemic in 2020, although spikes do clearly appear in other periods as well.

Based on the extracted data, we fit self-exciting models to three different time-series. First of all we have the series which consists of all of the jump-times we have extracted. Secondly, we consider only jump-times where log-returns are positive, and the third series contains jump-times with negative log-returns. In all cases the jump-size magnitudes equal the absolute value of the corresponding log-return, and thus the jump-sizes which feed into the maximum likelihood estimation are always positive. Note that in the maximum likelihood estimation we take the time and  magnitude of the jumps (i.e. the absolute value of the log return) into account via the step process, $U(t)$, which jumps whenever a jump occurs, and the jump size is given by the absolute value of the corresponding index log return.  

The log-likelihood of the model is well known (see Ogata~\cite{Og78} for asymptotic properties), if $\theta$ denotes the parameter vector to be found it is given by 
\begin{align*}
\log L_t(\theta) = \sum_{j=1}^2 \left(\sum_{T_k^{(j)} \leq t} \log \lambda_j (T_k^{(j)}) - \int_0^t \lambda_j(s) ds \right),
\end{align*}
where $\{T_k^{(j)}\}$, $j=1,2$ denote the jump times in the respective components.
\begin{figure} 
\centering
\includegraphics[width=6.5in, height=2.5in]{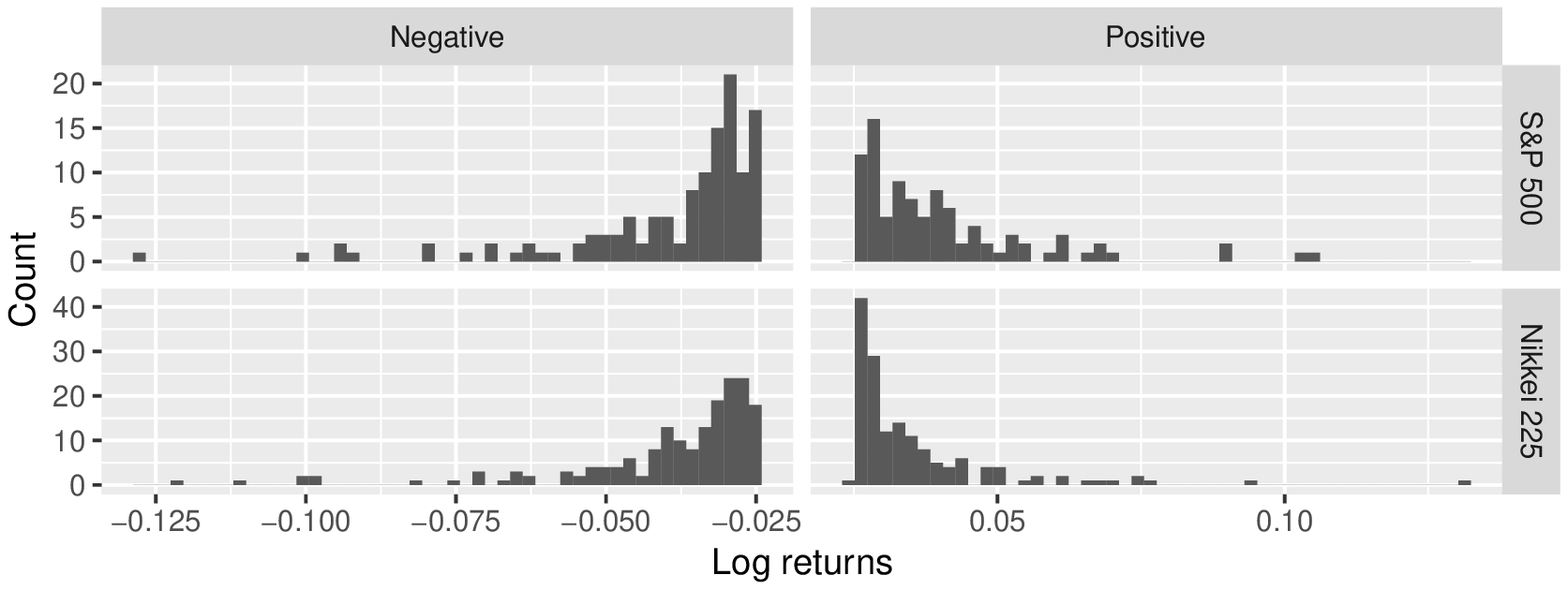}
\caption{Histograms of log returns the S\&P 500 (upper panel) and the Nikkei 225 (lower panel) stock indices, which are larger than $0.025$ in absolute value.}
\label{fig:lrLargeHists}
\end{figure}
Note that both the jump time and the magnitude (absolute value of log-return) of the corresponding jump feed into the maximum likelihood estimation via the specification of the intensity process.
\footnotesize
\begin{table}
\centering
\begin{tabular}{c|*{8}{c}}
Model & I & II & III & IV & V & VI & VII & VIII \\
\hline
 $\hat \lambda_{01}$ & 0.0074 & 0.0093 & 0.0053 & 0.0076 & 0.0076 & 0.0089 & 0.0082 & 0.0067  \\
 $\hat \lambda_{02}$ & 0.0216 & 0.0199 & 0.0211 & 0.0219 & 0.0198 & 0.0222 & 0.0178 & 0.0172 \\
 $\hat a_{11}$ & -0.0699 & -0.0672 & -0.0779 & -0.0723 & -0.0760 & -0.0643 & -0.5722 & -0.1190  \\
 $\hat a_{12}$ & -- & -0.0258 & -- & -0.0002 & -- & -0.0201 & -0.0194 & -- \\
 $\hat a_{21}$ & -- &  -0.0327 & -- & 0.0203 & -0.0337 & -- & -0.0308  & -0.0301 \\
 $\hat a_{22}$ & -0.0763 &  -0.0839 & -0.1019 & -0.0949 & -0.0722 & -0.0835 & -0.7631 & -0.1482 \\
 $\hat b_{11}$ &  1.5220 & 1.6117 & 1.5872 & 1.5736 & 1.6444 & 1.3913 & 1.8088 & 1.8064 \\
 $\hat b_{12}$ & -- & 0.3683 & 0.1331 & -- & -- & 0.3598 & 0.3055 & -- \\
 $\hat b_{21}$ & -- & 1.1952 & 0.8208 & -- & 1.1440 & -- & 1.4446 & 1.4729  \\
 $\hat b_{22}$ & 1.4908 & 1.3183 & 1.4234 & 1.5359 & 1.1566 & 1.6115 & 1.3476 & 1.2206  \\
 $\hat c$ & -- & -- & -- & -- & -- & -- & 0.1074 & 2.1541  \\
 $\hat d_1$ & -- & -- & -- & -- & -- & -- & 0.5086 & 0.0542 \\
 $\hat d_2$ & -- & -- & -- & -- & -- & -- & 0.6854 & 0.0813  \\
 \hline
 $LL$ & -1690 & -1667 & -1676 & -1687 & -1675 & -1682 & -1662 & -1669 \\
\hline
$\hat \lambda_{01}^+$ &  0.0038 &  0.0044 &  0.0033 & 0.0040 & 0.0039 & 0.0040 & 0.0043 & 0.0037 \\
$\hat \lambda_{02}^+$ & 0.0125  &  0.0108 &  0.0121 & 0.0125 & 0.0110 & 0.0125 & 0.0106 & 0.0107 \\
$\hat a_{11}^+$ &  -0.0528 &  -0.0520 &  -0.0547 &  -0.0544 &  -0.0557 & -0.0498 & -1.4179 & -0.1971 \\
$\hat a_{12}^+$ & -- & -0.0091 & -- & -0.0012 & -- & -0.0059 & -0.0084 & -- \\
$\hat a_{21}^+$ &  -- & -0.0306 & -- &  0.0274 & -0.0296 & -- & -0.0292 & -0.0285 \\
$\hat a_{22}^+$ & -0.0483 & -0.0412 & -0.0713 & -0.0751 & -0.0394 & -0.0490 & -1.0930 & -0.1518 \\
$\hat b_{11}^+$ &  1.1001 &  1.1477 &  1.1109 &  1.1479 &  1.1556 &  1.0476 & 1.2442 & 1.2211 \\
$\hat b_{12}^+$ &  -- &  0.1088 &  0.0392 &  -- &  -- &  0.0937 & 0.1083 & -- \\
$\hat b_{21}^+$ &  -- &  0.8521 &  0.6339 &  -- &  0.8241 &  -- & 0.9019 & 0.8738 \\
$\hat b_{22}^+$ &  0.9068 &  0.6514 &  0.9043 &  1.0162 &  0.6299 &  0.9177 & 0.6698 & 0.6336 \\
$\hat c^+$ & -- & -- & -- & -- & -- & -- & 0.0669 & 0.7371 \\
$\hat d_1^+$ & -- & -- & -- & -- & -- & -- & 1.3669 & 0.1447  \\
$\hat d_2^+$ & -- & -- & -- & -- & -- & -- & 1.0522 & 0.1139  \\
\hline
$LL^+$ & -953 & -942 & -945 & -950 & -943 & -953 & -942 & -943 \\
\hline  
$\hat \lambda_{01}^-$ & 0.0068 & 0.0054 & 0.0043 & 0.0063 & 0.0068 & 0.0072 & 0.0038 & 0.0061 \\
$\hat \lambda_{02}^-$ & 0.0186 & 0.0195 & 0.0290 & 0.0190 & 0.0255 & 0.0203 & 0.0190 & 0.0246 \\
$\hat a_{11}^-$ & -0.0599 & -0.0114 & -0.0811 & -0.0717 & -0.0599 & -0.0563 & -0.5389 & -4.0027  \\
$\hat a_{12}^-$ &  -- & -0.1233 & -- & 0.0068 & -- & -0.0421 & 0.1958 & -- \\
$\hat a_{21}^-$ &  -- &  0.0120 & -- & 0.0517 & 0.4209 & -- & 0.6158 & 0.4483 \\
$\hat a_{22}^-$ & -0.0658 & -0.1885 & -0.9272 & -0.1298 & -1.4496 & -0.0956 & -8.1084 & -9.4716 \\
$\hat b_{11}^-$ &  1.1920 & 1.6858 & 1.3095 & 1.3158 & 1.1924 & 1.0760 & 0.0001 & 1.3725 \\
$\hat b_{12}^-$ &  -- & 0.6901 & 0.3349 & -- & -- & 0.5814 & 0.7452 & -- \\
$\hat b_{21}^-$ &  -- & 1.9125 & 10.3821 & -- & 12.3842 & -- & 10.5548 & 12.5136  \\
$\hat b_{22}^-$ & 0.9497 & 1.0751 & 0.0000 & 1.1035 & 0.0000 & 1.2756 & 0.0001 & 0.0001  \\
$\hat c^-$ & -- & -- & -- & -- & -- & -- & 0.1400 & 0.0417 \\
$\hat d_1^-$ & -- & -- & -- & -- & -- & -- & 0.3233 & 3.9493  \\
$\hat d_2^-$ & -- & -- & -- & -- & -- & -- & 6.9854 & 8.0846  \\
\hline
$LL^-$ & -1168 & -1137 & -1141 & -1164 & -1140 & -1158 & -1124 & -1138  \\
\hline
\end{tabular}
\caption{Maximum likelihood estimates of the linear model and non-linear models, with corresponding log-likelihoods, where the entry -- means that the corresponding coefficient is set equal to $0$. Estimates corresponding to all jumps, positive jumps and negative jumps are displayed at the top, middle and bottom of the table, respectively.} \label{tab:MLELinear}
\end{table}
\normalsize
We fit a bivariate linear model \eqref{ex:linear} to the data and a non-linear generalization of the linear model. The model is fitted under the constraint that the stability conditions of the present paper is fulfilled. For the linear model, we assume that the dynamics are on the form $\lambda(t) = (\lambda_1(t), \lambda_2(t))^\top$, where $\lambda(t)$ is given by the linear model \eqref{ex:linear}, with
\begin{align}\label{def:LinParam}
\lambda_0 = \begin{pmatrix}
\lambda_{01} \\
\lambda_{02}
\end{pmatrix}, \ \
A = \begin{pmatrix}
a_{11} & a_{12} \\
a_{21} & a_{22}
\end{pmatrix}, \ \text{and }
B = \begin{pmatrix}
b_{11} & b_{12} \\
b_{21} & b_{22}
\end{pmatrix},
\end{align} 
and the intensity $\lambda_1(t)$ corresponds to the S\&P, while $\lambda_2(t)$ is the Nikkei intensity. The non-linear model has an intensity on the form 
\begin{equation*}
d\lambda(t) = (A + D\exp(-c\|\lambda(t)\|_2^2))(\lambda(t) - \lambda_0)dt + BdU(t),
\end{equation*}
where $\|x\|_2^2 = x^\top x$, for any $x \in \R^2$ denotes the Euclidian norm in $\R^2$, $\lambda_0$, $A$ and $B$ are given by \eqref{def:LinParam}, $c > 0$ is a constant and 
$$
D = \begin{pmatrix}
d_{1} & 0 \\
0 & d_{2}
\end{pmatrix}
$$
is a diagonal matrix. Notice that that the speed of mean reversion (i.e. the behaviour of the bivariate intensity between jumps) is allowed to change with the norm of $\lambda(t)$. Thus, the speed varies with the overall intensity level, there is a built in regime change in the model, the speed of mean reversion is different in times of high intensity than in times of low intensity. The non-linear model is an extension of the linear model, in the sense the linear model is recovered when $D=0$.

In table \ref{tab:MLELinear} the results of our maximum likelihood estimation are displayed in eight different cases, for all jumps, positive jumps and negative jumps, respectively. The first six cases correspond to a linear intensity which has been fitted to the data by forcing different off-diagonal parameters of $A$ and $B$ to be zero. That is, we fit six distinct linear models to the data, where the models differ on which off-diagonal elements are non-zero. 

From inspecting the parameter estimates and the corresponding log-likelihoods we see that out of the first six linear models, model II produces the highest log-likelihood in all cases. Model II sets no parameters equal to zero. Out of the remaining models, model V comes closest to model II in terms of the maximum likelihood, and model III produces the third best fit. A common feature of these three linear models is that the all allow component one to influence component two, i.e. they allow the S\&P 500 intensity to influence the Nikkei 225 intensity. While models I, IV and VI do not possess that feature. Model V has $A$ and $B$ as lower triangular matrices. A lower triangular $B$ means that a jump in the S\&P 500 index will cross-excite the jump intensity of the Nikkei 225 index but not vice versa, and a lower triangular $A$ means that the values of both intensities contribute to the mean reversion of the Nikkei intensity, whereas only the value of the S\&P index contributes to the mean reversion of the S\&P jump intensity. Model III on the other hand has a diagonal $A$ matrix, with diagonal values that dictate the speed of mean reversion for the respective jump intensities, and $B$ has non-zero entries, meaning that jumps contribute to both self- and cross-excitation of jump intensities between markets. 

It is interesting to note that, for the three different series, under models II and III it holds that $\hat b_{21} > \hat b_{12}$, so a jump in the S\&P 500 index typically causes more excitation in the Nikkei 225 index than vice versa, the difference is especially high in the case of negative jumps under model III. Which means that a negative return in the S\&P 500 index is much more likely to cause a negative return in the Nikkei 225 index than vice versa. In fact, we observe that the case of negative jumps this causality relationship is much stronger than in the case of positive jumps or all jumps. 


In models VII and VIII we have extended models II and V, respectively, to non-linear models, since models II and V had the highest log likelihood among the linear models. From inspecting the parameters of the non-linear models we see that their main feature is that the speed of mean reversion stronger when the intensities are high and slower when they are low. Note in particular that this effect is very strong in the case of negative jumps of the Nikkei 225 index, which is seen from the fact that the difference between $\hat a_{22}^-$ and $\hat d_2^-$ is quite large, by far the largest among the non-linear models. We also note that, the gain in switching to a non-linear model in terms of log-likelihoods is highest fo negative jumps. Thus, in that case we observe a regime change in times of crises, which means that the speeds of mean-reversion become higher than when the intensities increase.

\begin{figure} 
\centering
\includegraphics[width=6.5in, height=6in]{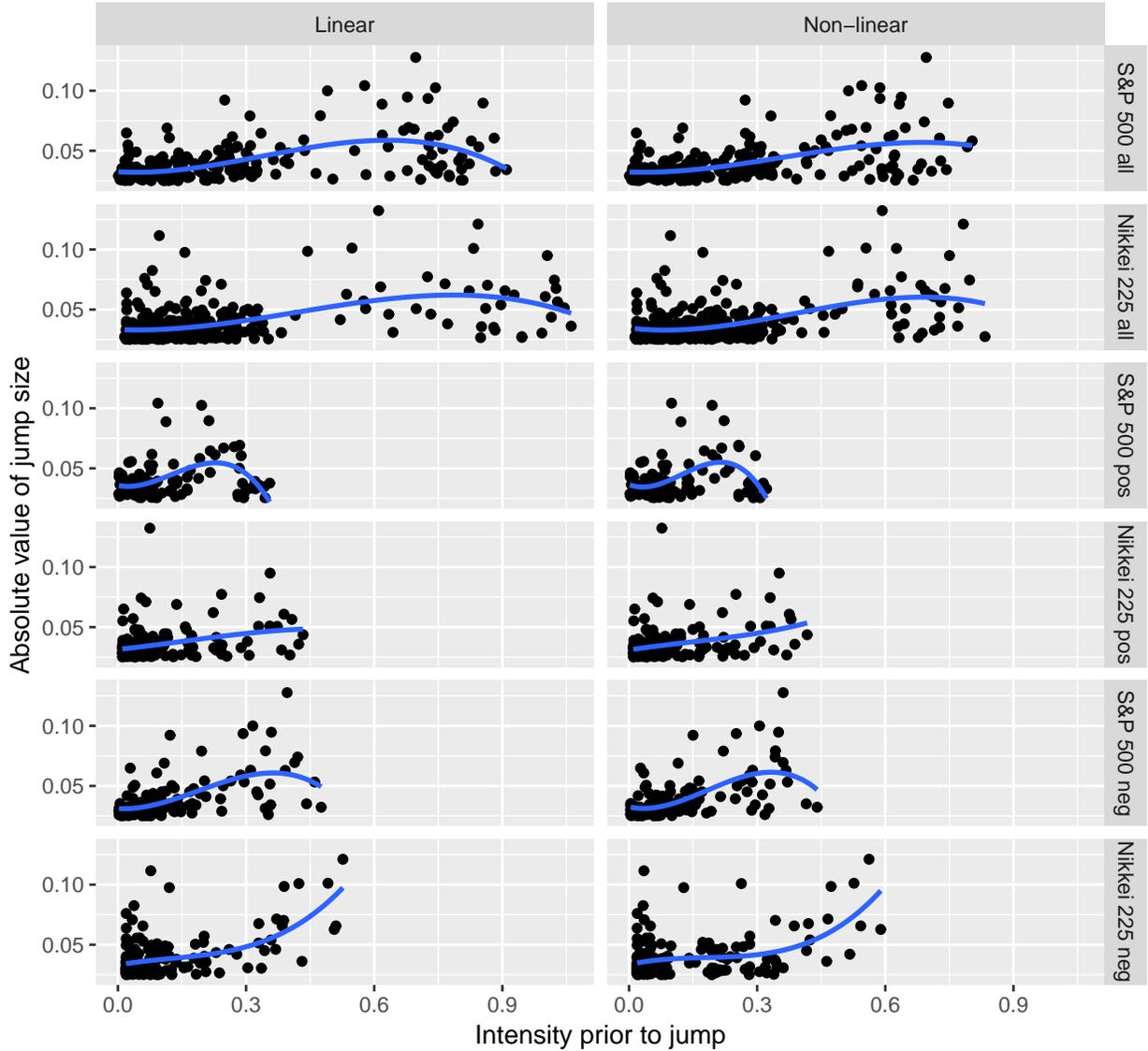}
\caption{Intensity of the respective stock indices S\&P 500 and the Nikkei 225, versus the absolute value of the corresponding jump sizes, when all jumps, positive jumps and negative jumps are considered, respectively. The parameters of the intensities are given by models II (linear model) and VIII (non-linear model) respectively in table \ref{tab:MLELinear}. B-splines have been fitted to the data. }
\label{fig:intensity}
\end{figure}

In figure \ref{fig:intensity} we have plotted intensity values immediately prior to jumps against absolute values of the corresponding log returns for the respective indices, where the parameters of the intensities are given by model II for the linear model and model VIII for the non-linear model. From the plots we see first of all that it is valid to assume that jump sizes depend on intensity values as we have assumed in our model, and secondly that the relationship between intensities and jump sizes depend on the market and sign of the jumps. When all jumps are considered the average jump sizes seem to flatten out or decrease for high intensity values. When only positive or negative jumps are considered for the S\&P 500 index, there is initially a positive relationship between the intensity values and jump-sizes, but the jump-sizes clearly become smaller for the highest intensity values. Note from figure \ref{fig:intensity_series} we see that the highest intensity values appear in clusters around times of crises. Thus, we can say that in times of crises the average jump-sizes become somewhat smaller than at the beginning of (what might become) a crisis, in other words there is an asymmetry of ascent and descent of clusters of large jumps. At the top of the crisis, since the model suggests that the S\&P 500 index is the leading index out the two, this effect contributes to, eventually, pushing the intensity back towards a calmer level.

On the other hand, when only positive or negative jumps are considered for the Nikkei 225 index, such an effect is less clear. Indeed, for positive jumps, figure \ref{fig:intensity} suggests there is a positive linear relationship between the intensity and jump-sizes, and for negative jumps, the data suggests a positive quadratic relationship. The reason why this does not cause an even higher intensity than observed is first of all that the speed of mean reversion for the jump intensities increases as the jump intensity increases, as we observed from the maximum-likelihood estimation of the non-linear Nikkei 225 component and secondly that jumps in the Nikkei 225 component does not raise the overall bivariate intensity as much as jumps in the S\&P 500 component. As previously noted, this shift in mean-reversion was strongest for the Nikkei 225 component when negative jumps were considered. 

\section{Conclusion}\label{sec:Conc} 
In the present paper we have defined a class of Markovian self- and cross-exciting processes. We have given stability conditions, and discussed the linear case in some detail. Finally we fitted our model to a bivariate time-series, which was extracted from large movements in the S\&P 500 and Nikkei 225 indices respectively. From our study we concluded that a non-linear variant of our model fitted the data best, and that the S\&P 500 index is leading in the sense that big movements there cause more cross-excitation in the Nikkei 225 index than vice versa.

\begin{figure} 
\centering
\includegraphics[width=6.5in, height=6in]{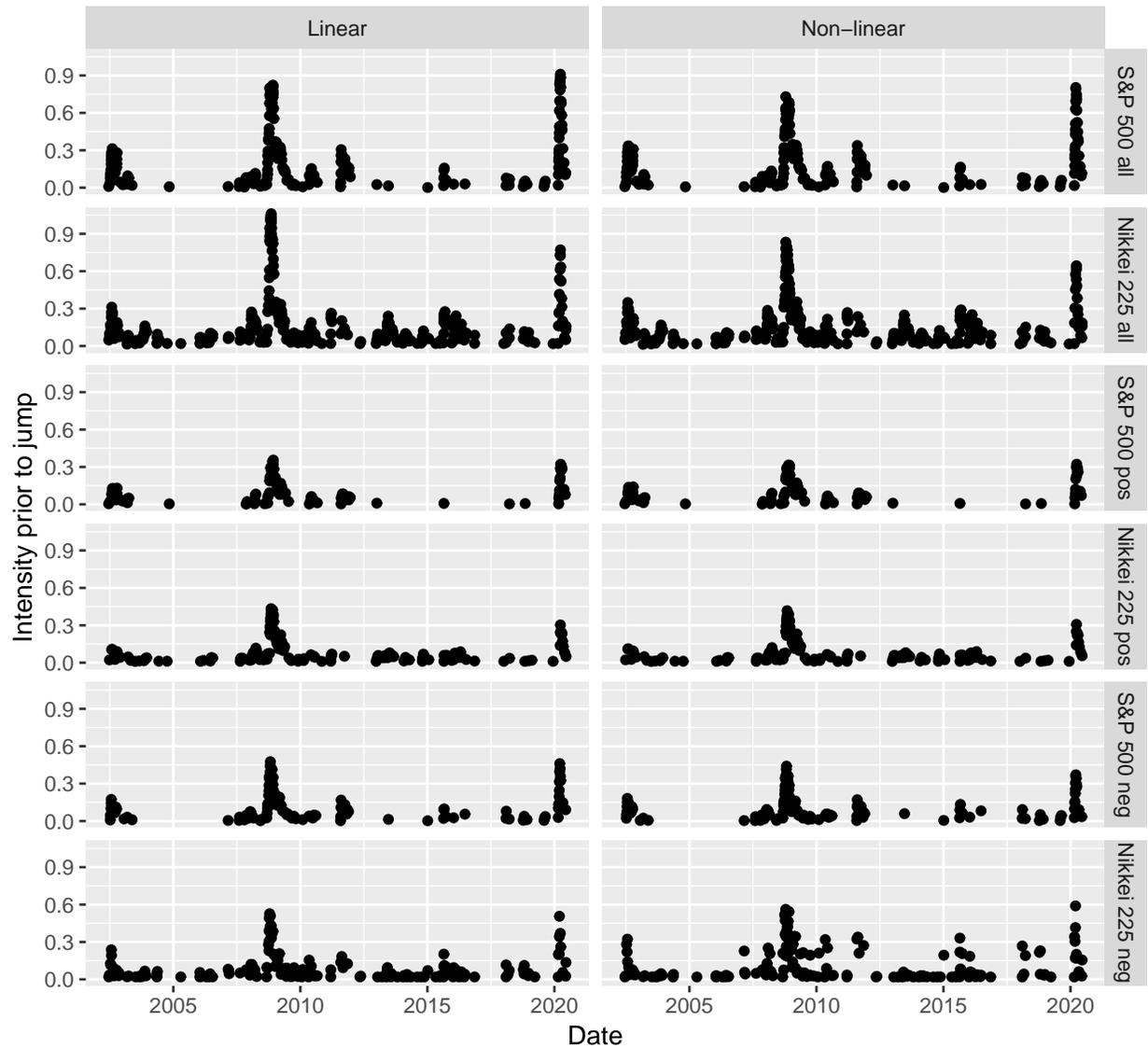}
\caption{Values of the Intensity time series of the respective stock indices S\&P 500 (upper panel) and the Nikkei 225 (lower panel) immediately before jumps, when all jumps, positive jumps and negative jumps are considered, respectively. The parameters of the intensities are given by models II (linear model) and VIII (non-linear model) respectively in table \ref{tab:MLELinear}.}
\label{fig:intensity_series}
\end{figure}

\printbibliography

\end{document}